\documentclass[a4paper,fleqn]{cas-sc}

\usepackage[square,numbers,sort&compress]{natbib}

\usepackage{graphicx}          % Include this line if your
\usepackage{lineno,hyperref}
\usepackage{amsfonts}
\usepackage{amsmath}
\usepackage{amssymb}
\usepackage{amscd,multirow}
\usepackage{graphicx}
\usepackage{epstopdf}
\usepackage[titletoc]{appendix}
\usepackage{subfigure}
\usepackage[justification=centering]{caption}
\usepackage{subeqnarray,cases}
\usepackage[ruled]{algorithm2e}
\numberwithin{equation}{section}

\allowdisplaybreaks[4]
\newtheorem{theorem}{Theorem}

\newtheorem{lemma}{Lemma}

\newtheorem{remark}{Remark}

\newtheorem{assumption}{Assumption}

\newenvironment{proof}{\noindent {\textbf{Proof.}}}

\usepackage{fancyhdr} 
\begin{document}
\let\WriteBookmarks\relax
\def\floatpagepagefraction{1}
\def\textpagefraction{.001}
\let\printorcid\relax

% Short title
\shorttitle{Non-ergodic convergence rate of inertial accelerated  primal-dual algorithm} 

% Short author
\shortauthors{X. He, N.J. Huang, Y.P. Fang}  

% Main title of the paper
\title[mode = title]{Non-ergodic convergence rate of an inertial accelerated  primal-dual algorithm  for saddle point problems}  

% Title footnote mark
% eg: \tnotemark[1]
%\tnotemark[1] 
%
%% Title footnote 1.
%% eg: \tnotetext[1]{Title footnote text}
%\tnotetext[1]{Second-order primal-dual dynamic} 

% First author
%
% Options: Use if required
%\author[1]{Xin He}\ead{hexinuser@163.com}
%\author[2]{Rong Hu}\ead{ronghumath@aliyun.com}
%\author[3]{Ya-Ping Fang}\ead{ypfang@scu.edu.cn} 
%\cormark[1]
%\cortext[1]{Corresponding author}
\author[1]{Xin He}\ead{hexinuser@163.com}
\author[2]{Nan-Jing Huang}\ead{njhuang@scu.edu.cn}
\author[2]{Ya-Ping Fang}\ead{ypfang@scu.edu.cn} 
\cormark[1]
\cortext[1]{Corresponding author}

\address[1]{School of Science, Xihua University, Chengdu, Sichuan, P.R. China}
    
\address[2]{Department of Mathematics, Sichuan University, Chengdu, Sichuan, P.R. China}

\nonumnote{}

% Here goes the abstract
\begin{abstract}
In this paper, we design an inertial accelerated  primal-dual algorithm to address the convex-concave saddle point problem, which is formulated as $\min_{x}\max_{y} f(x) + \langle Kx, y \rangle - g(y)$. Remarkably, both functions $f$ and $g$ exhibit a composite structure, combining ``nonsmooth'' + ``smooth'' components. Under the assumption of partially strong convexity in the sense that $f$ is convex and $g$ is strongly convex,  we introduce a novel inertial accelerated primal-dual algorithm based on  Nesterov's extrapolation. This algorithm can be reduced  to two classical accelerated forward-backward methods for unconstrained optimization problem. We show that the proposed algorithm achieves a  non-ergodic $\mathcal{O}(1/k^2)$ convergence rate, where $k$ represents the number of iterations. Several numerical experiments validate the efficiency of our proposed algorithm.
\end{abstract}

\begin{keywords}
Inertial accelerated primal-dual algorithm\sep  Saddle point problem\sep Non-ergodic rate\sep  Nesterov acceleration
\end{keywords}

\maketitle

\section{Introduction}
In this paper, we consider the following min-max saddle point problem:
\begin{equation}\label{ques}
\min_{x\in\mathbb{R}^n}\max_{y\in\mathbb{R}^m}\mathcal{L}(x,y) = f(x) + \langle Kx, y \rangle - g(y)
\end{equation}
with
\[f(x)=f_1(x)+f_2(x)\quad \text{and}\quad g(y)=g_1(y)+g_2(y).\]
Here, $K\in\mathbb{R}^{m\times n}$, $f_1:\mathbb{R}^{n}\to \mathbb{R}\cup\{+\infty\}$ and $g_1:\mathbb{R}^{m}\to\mathbb{R}\cup\{+\infty\}$ are proper, closed and convex functions,  $f_2:\mathbb{R}^{n}\to \mathbb{R}$ and $g_2:\mathbb{R}^{m}\to\mathbb{R}$ are smooth and convex functions. Problem \eqref{ques} finds applications across diverse fields, such as machine learning, image processing, computer vision, and the finding  a saddle point for the Lagrangian function in convex minimization problems (see \cite{Goldstein2015, ChambolleP, LiML2019ACC, bauschke2019, bubeck}).

For solving problem \eqref{ques}, Chambolle and Pock \cite{ChambolleP} introduced the first-order primal-dual algorithm (PDA) defined by the following iteration scheme: 
\begin{numcases}{}
	x_{k+1} = {\bf Prox}_{\alpha, f}(x_k-\alpha K^T y_k),\nonumber\\
	\bar{x}_{k+1} = x_{k+1}+\theta(x_{k+1}-x_{k}),\label{pda}\\
	y_{k+1} = {\bf Prox}_{\beta, g}(y_k+\beta K \bar{x}_{k+1}).\nonumber
\end{numcases}
Here, the $\text{\bf Prox}_{\alpha, f}$ denotes the proximal operator for the function $f$, defined as:
\[\text{\bf Prox}_{\alpha,f}(x) = \mathop{\arg\min}_y f(y)+\frac{1}{2\alpha}\| y-x\|^2\]
with $\alpha>0$. This algorithm has gained significant attention due to its effectiveness in solving various imaging problems. When $\theta=0$, the PDA reduces to the classical Arrow-Hurwicz method \cite{Arrow}. Chambolle and Pock \cite{ChambolleP} demonstrated that the PDA is closely related to the extra-gradient method \cite{Korpelevich}, Douglas-Rachford splitting method \cite{Lions}, and preconditioned alternating direction method of multipliers \cite{Esser}. They also established that the PDA with $\theta \in\{0,1\}$ achieves an ergodic $\mathcal{O}(1/k)$ convergence rate when both $f$ and $g$ are convex. Additionally, for suitable choices of $\alpha$, $\beta$, and $\theta$, they proved that the PDA attains an ergodic linear convergence when $f$ and $g$ are strongly convex. In the case where $g$ is $\mu_g$-strongly convex, Chambolle and Pock \cite{ChambolleP} introduced an accelerated PDA with adaptive parameters, which enjoys an ergodic $\mathcal{O}(1/k^2)$ convergence rate. Chen et al. \cite{ChenSiam2014} proposed an accelerated primal-dual method for problem \eqref{ques}, which achieves an ergodic convergence with a rate of $\mathcal{O}\left({L_f}/{k^2}+{\|K\|}/{k}\right)$,
where $L_f$ represents a Lipschitz constant of $\nabla f$. By introducing the Bregman distance, Chambolle and Pock \cite{Chambolle2016MP} proposed a first-order primal-dual algorithm for solving problem \eqref{ques} when $f=f_1+f_2$. They established ergodic convergence rates with simpler proofs compared to \cite{ChambolleP}. He et al. \cite{He2014siamPD,HeJMIV} introduced primal-dual hybrid gradient methods, which exhibit an ergodic convergence with a rate of $\mathcal{O}(1/k)$ in the convex case. Tran-Dinh \cite{Tran2022} have designed  a unified convergence analysis framework for  the accelerated smoothed gap reduction algorithm proposed in \cite{Tran2018}.  They demonstrate that the proposed algorithms enjoy a non-ergodic $\mathcal{O}(1/k^2)$ convergence rate in the partially strongly convex case. Zhu et al. \cite{Zhu2022} have introduced  novel primal-dual algorithms aimed at solving a class of nonsmooth and nonlinear convex-concave minimax problems, and the algorithms demonstrate both ergodic and non-ergodic $\mathcal{O}(1/k^2)$ convergence rates in the partially strongly convex assumption. In the realm of inexact first-order primal-dual algorithms for solving problem \eqref{ques}, various approaches have been explored, as documented in \cite{RaschCH2020, Jiang2021NA, Jiang2021}. These inexact methods achieve ergodic convergence rates of $\mathcal{O}(1/k)$ in the convex case, $\mathcal{O}(1/k^2)$ in the case of partially strong convexity (when either $f$ or $g$ is strongly convex). Furthermore, numerous variants of the primal-dual algorithm have emerged, such as adaptive primal-dual splitting methods in \cite{Goldstein2015}, randomized coordinate-descent methods in \cite{Fercoq}, and primal-dual methods with linesearch in \cite{MalitskySIAM}. For additional variations and developments in primal-dual methods, we refer the reader to \cite{Bai2023,bai2020several,HeAMO2024,chang2021,hamedani2021,Mokhtari2020,Tan2020,Boct2023COA}.

As mentioned above, various ergodic convergence rate results have been  established for  primal-dual algorithms for problem \eqref{ques} in the literature. Only \cite{Tran2022,Zhu2022} investigated the non-ergodic convergence of primal-dual algorithms, but they did not consider the problem \eqref{ques} with $f$ and $g$ having a composite structure. It is worth mentioning that multiple adaptive parameters of the algorithm considered in \cite{Tran2022,Zhu2022}  need to be set, which introduce challenges in parameter selection during actual numerical calculations. For the primal-dual algorithm in \cite{Tran2022}, a fixed proximal center $\dot{y}$ is required in each iteration, which makes it impossible to make more efficient use of the information in the iteration sequence.  In recent years, numerous researchers have devoted their efforts to studying the Augmented Lagrangian Method (ALM) for solving:
\begin{equation}\label{ques_sepA}
				\min_{x\in  \mathbb{R}^{m}} \quad f(x)\qquad s.t.  \  Ax=b,
	\end{equation}
and the Alternating Direction Method of Multipliers (ADMM) for solving:
	\begin{equation}\label{ques_sep}
				\min_{x\in  \mathbb{R}^{m},y\in  \mathbb{R}^{n}} \quad f(x)+g(y) \qquad s.t.  \  Ax+By=b.
	\end{equation}
Significant efforts have been made to explore the non-ergodic convergence rates of various variants of ALM and ADMM. In the convex case, non-ergodic $\mathcal{O}(1/k)$ convergence rates for ALMs and ADMMs have been investigated in \cite{HeY2015NM, Li2019}. Some accelerated ALMs (in the convex case) and accelerated ADMMs (in the partially strongly convex case), incorporateed with  Nesterov's extrapolation technique from \cite{Nesterov1983, Nesterov2018, Tseng2008, BeckIma},  have been proposed in \cite{HeNA, Xu2017, BotMP2022, Tran2018, HeAuto, Tran2020, LuoSco,Luo2021Accer}. These methods achieve non-ergodic $\mathcal{O}(1/k^2)$ convergence rates for both  objective residual and  feasibility violation. It is well-known that problems \eqref{ques_sepA} and \eqref{ques_sep} can be equally reformulated in the form of problem \eqref{ques}.

In this paper, we  aim to   apply  acceleration techniques  inspired by ALMs and ADMMs, known for the non-ergodic convergence rates, to develop a novel primal-dual algorithm with the non-ergodic convergence for addressing problem \eqref{ques}.  By incorporating classical inertial coefficients and introducing additional simple constant parameters, we propose the following inertial accelerated  primal-dual algorithm (Algorithm \ref{al1}) for problem \eqref{ques}, where $\mu_g>0$ represents the strongly convex coefficient of $g$.

  \begin{algorithm}
       \caption{Inertial accelerated  primal-dual algorithm for problem \eqref{ques}: Partially strongly convex case}
        \label{al1}
        {\bf Initialization:} Choose $u_1= x_1 =x_0\in dom(f),\ v_1=v_0=y_1 =y_0\in dom(g)$. Pick $\alpha,\beta>0,  t_1\geq 1$. \\
        \For{$k = 1, 2,\cdots$}{
        Set $t_{k+1} = \min\left\{\frac{1+\sqrt{1+4t_k^2}}{2}, \sqrt{t_k^2+\mu_g\beta t_k}\right\}$. \\
        Compute
        \begin{eqnarray*}
        	&&(\bar{x}_k,\bar{y}_k) = (x_k,y_k) +\frac{t_k-1}{t_{k+1}}[(x_{k},y_k)-(x_{k-1},y_{k-1})].\label{eq_al1_1} 
        \end{eqnarray*}\\
        Update $(x_{k+1},u_{k+1})$ using one of the following two options:\\
        \qquad {\bf Option 1:}
          \begin{eqnarray*}
        	&&x_{k+1} = {\bf Prox}_{\alpha, f_1}\left(\bar{x}_k-\alpha \left(\nabla f_2(\bar{x}_k)+K^T\left(v_k+\frac{t_k}{t_{k+1}}(v_k-v_{k-1})\right)\right)\right)	.\label{eq_al1_3}\\
       		&& u_{k+1} = x_{k+1}+(t_{k+1}-1)(x_{k+1}-x_k).\label{eq_al1_4}
        \end{eqnarray*}\\
        \qquad {\bf Option 2:}
         \begin{eqnarray*}
         && u_{k+1} = {\bf Prox}_{\alpha t_{k+1}, f_1}\left(u_k-\alpha t_{k+1}\left(\nabla f_2(\bar{x}_k)+K^T\left(v_k+\frac{t_k}{t_{k+1}}(v_k-v_{k-1})\right)\right)\right).\label{eq_al2_3}\\
      	&&	x_{k+1} = \frac{t_{k+1}-1}{t_{k+1}}x_{k} + \frac{1}{t_{k+1}}u_{k+1}\label{eq_al2_5}.
         \end{eqnarray*}  		 
         Update $(y_{k+1},v_{k+1})$ by 
           \begin{eqnarray*}
       		&& v_{k+1} = {\bf Prox}_{\frac{\beta}{t_{k+1}}, g_1}\left(v_k-\frac{\beta}{t_{k+1}}(\nabla g_2(\bar{y}_k)-K u_{k+1})\right).\label{eq_al1_5}\\	
      		&& y_{k+1} = \frac{t_{k+1}-1}{t_{k+1}}y_k + \frac{1}{t_{k+1}}v_{k+1}.\label{eq_al1_6}
      \end{eqnarray*}
\If{A stopping condition is satisfied}{Return $(x_{k+1},y_{k+1})$.}
}
\end{algorithm}

It's noteworthy that when $K=0$, the subproblem updates for functions $f$ and $g$ become independent. In this scenario, the update of $x_{k+1}$  in {\bf Option 1} transforms into the following accelerated forward-backward algorithm:
\begin{equation}\label{FISTA}
	\begin{cases}
		 \bar{x}_{k} = x_{k}+\frac{t_k-1}{t_{k+1}}(x_{k}-x_{k-1}),\\
		x_{k+1} ={\bf Prox}_{\alpha, f_1}(\bar{x}_k-\alpha \nabla f_2(\bar{x}_k)),
	\end{cases}
\end{equation}
where $t_{k+1} = (1+\sqrt{1+4t_k^2})/2$. This algorithm, proposed by Beck and Teboulle \cite{BeckIma}, is also known as the fast iterative shrinkage-thresholding algorithm (FISTA).  Similarly,  The update of $x_{k+1}$ and $u_{k+1}$ in  {\bf Option 2} transforms into the accelerated forward-backward algorithm with Tseng's scheme \cite{Tseng2008}:
\begin{equation}\label{NestSS}
	\begin{cases}
		 \bar{x}_{k} = x_{k}+\frac{t_k-1}{t_{k+1}}(x_{k}-x_{k-1}),\\
		u_{k+1} ={\bf Prox}_{\alpha t_{k+1}, f_1}(u_k-\alpha t_{k+1}\nabla f_2(\bar{x}_k)),\\
		x_{k+1} = \frac{t_{k+1}-1}{t_{k+1}}x_k+	\frac{1}{t_{k+1}}u_{k+1}	.
	\end{cases}.
\end{equation}
These two algorithms are designed to solve the composite problem $\min_x f_1(x)+f_2(x)$. 

To establish the non-ergodic convergence rate of Algorithm \ref{al1}, we introduce the following assumption for problem \eqref{ques}. Under this assumption, we will demonstrate the non-ergodic $\mathcal{O}(1/k^2)$ convergence rate of Algorithm \ref{al1}, and provide numerical examples to validate our theoretical findings.

\begin{assumption}\label{assA}
	 Suppose that $f(x)=f_1(x)+f_2(x)$ and $g(y)=g_1(y)+g_2(y)$, where $f_1$  is a proper, closed and convex function; $g_1$ is a proper, closed and $\mu_g$-strongly convex function with $\mu_g> 0$,  $f_2$ is a convex function and  has an $L_{f_2}$-Lipschitz continuous gradient, $g_2$ is a  convex function and  has an $L_{g_2}$-Lipschitz continuous gradient; the saddle point set $\Omega$ is nonempty.
\end{assumption}
Then, for any $y_1,y_2 \in\mathbb{R}^m$, and  $\widetilde{\nabla} g_1(y_1)\in\partial g_1(y_1)$, we have
\begin{eqnarray*}\label{strongg}
	g_1(y_2) - g_1(y_1)-\langle  \widetilde{\nabla} g_1(y_1), y_2-y_1\rangle\geq \frac{\mu_g}{2}\|y_2-y_1\|^2. 
\end{eqnarray*}
Note that: The smooth terms $f_2$ and $g_2$ can vanish in problem \eqref{ques}. When $f_2(x)\equiv 0$,  it is a convex function and $\nabla f_2(x)$ is $L_{f_2}$-Lipschitz continuous for any $L_{f_2}>0$. Similarly, When $g_2(y)\equiv 0$, it is a convex function and $\nabla g_2(y)$ is $L_{g_2}$-Lipschitz continuous for any $L_{g_2}>0$.

The structure of this paper is organized as follows: In Section 2, we introduce fundamental concepts and summarize basic lemmas for further analysis. Section 3 is dedicated to the study of convergence rate of Algorithm \ref{al1}. Section 4 provides  numerical examples to verify our theoretical results. Finally, we offer concluding remarks in Section 5.

\section{Preliminaries}
In this section, we will introduce some basic notations and preliminary lemmas.

Let $\langle \cdot, \cdot \rangle$ and $\|\cdot\|$ represent the inner product and the Euclidean norm, respectively. For a  function  $f:\mathbb{R}^{n}\to\mathbb{R}\cup\{+\infty\}$, the domain of $f$ is defined as $dom(f)=\{x\in\mathbb{R}^n | f(x)<+\infty\}$. We say that  $f$ is proper if $dom(f) \neq \emptyset$, and that  $f$ is closed if $f(x) \leq \liminf_{y\to x}f(y)$ and $dom(f)$ is closed. For a proper, closed and convex function $f:\mathbb{R}^{n}\to\mathbb{R}\cup\{+\infty\}$, the domain of $f$ is a closed and convex set. The subdifferential of $f$ at $x$ is defined as:
 \[\partial f(x) = \{\omega \in\mathbb{R}^n | f(y)\geq f(x)+\langle \omega,y-x \rangle,\quad \forall y\in\mathbb{R}^n\},\]
and we denote  $\widetilde{\nabla} f(x)\in\partial f(x)$ to be a subgradient of $f$ at $x$. 
 
 We denote the saddle point set of problem \eqref{ques} as $\Omega$. For any $(x^*,y^*)\in\Omega$, we have
\begin{equation}\label{saddle}
	 \mathcal{L}(x^*,y)\leq \mathcal{L}(x^*,y^*)\leq \mathcal{L}(x,y^*).
\end{equation}
This implies 
\begin{equation}\label{saddle_AA}
	-K^T y^*\in \partial f(x^*) ,\quad Kx^*\in \partial g(y^*).
\end{equation}

Recall the following partial primal-dual gap (introduced in \cite{ChambolleP}):
\[ 
	\mathcal{G}_{\mathcal{B}_1\times \mathcal{B}_2}(x,y) = \max_{\bar{y}\in \mathcal{B}_2} \mathcal{L}(x,\bar{y})-  \min_{\bar{x}\in \mathcal{B}_1} \mathcal{L}(\bar{x},y),
\]
where $\mathcal{B}_1\times \mathcal{B}_2$ is a compact subset of $\mathbb{R}^n\times\mathbb{R}^m$  which contains a saddle point of  problem \eqref{ques}. 
It follows from \eqref{saddle} that $	\mathcal{G}_{\mathcal{B}_1\times \mathcal{B}_2}(x,y) \geq \mathcal{L}(x,y^*) -\mathcal{L}(x^*,y) \geq 0$ with $(x^*,y^*)\in\Omega$. If $\mathcal{G}_{\mathcal{B}_1\times \mathcal{B}_2}(x^*,y^*) =0$, with $(x^*, y^*)$ lying in the interior of $\mathcal{B}_1\times \mathcal{B}_2$, we can conclude that $(x^*, y^*)\in\Omega$.

Next, we will revisit the fundamental lemmas that will be used later.

\begin{lemma}\label{le_S1}
For any $x,y,z \in\mathbb{R}^n$, and $r>0$, the following equalities hold:
\begin{eqnarray}
	&& \frac{1}{2}\|x\|^2-\frac{1}{2}\|y\|^2=\langle x,x-y\rangle-\frac{1}{2}\|x-y\|^2,\label{eq_know}\\
	 &&\langle x,y\rangle\leq \frac{1}{2}\left(r\|x\|^2+\frac{1}{r}\|y\|^2\right).\label{eq_know1}
\end{eqnarray}
\end{lemma}

From \cite{Nesterov1983} and \cite[Lemma A.3]{HeADMM}, we can get the following result.
\begin{lemma}\label{le_S2}
	The positive sequence $\{t_k\}_{k\geq 1}$ generated by $t_{k+1} = \min\left\{\frac{1+\sqrt{1+4t_k^2}}{2}, \sqrt{t_k^2+a t_k}\right\}$ with $a>0$ and $t_1\geq 1$ satisfies $t_{k}\geq \min\{\frac{1}{2},b\}(k+1)$, where $b = \frac{2a t_1}{a+4t_1}$.
\end{lemma}

\begin{lemma}\label{le_S3}
	If $f:\mathbb{R}^{n}\to\mathbb{R}$ is a convex function and has  a Lipschitz continuous gradient with constant $L_{f}$, then for any $x,y,z\in\mathbb{R}^{m}$,  we have
	\begin{eqnarray*}\label{eq_lip}
   \langle \nabla f(z), x-y\rangle\geq   f(x)-f(y)-\frac{L_f}{2}\|x-z\|^2.
  \end{eqnarray*}
\end{lemma}
\begin{proof} 
Since $f$ has a Lipschitz continuous gradient, it follows from \cite[Theorem 2.1.5]{Nesterov2018} that
\[f(y) -f(x)-\langle \nabla f(x),y-x\rangle \leq \frac{L_{f}}{2}\|y-x\|^2, \quad \forall x,y\in\mathbb{R}^m.\]
This together with the convexity of $f$ implies
\begin{eqnarray*}
    f(x)-f(y) &=& f(x)-f(z)+f(z)-f(y)\\
			 &\leq& \langle \nabla f(z), x-z\rangle+\frac{L_f}{2}\|x-z\|^2+ \langle \nabla f(z), z-y\rangle\\
			 &=& \langle\nabla f(z), x-y\rangle+\frac{L_f}{2}\|x-z\|^2.
\end{eqnarray*} 
It yields the result.
\end{proof}

\begin{lemma}\cite[Lemma 4]{HeAuto} \label{le_S4}
	Let $\{h_k\}_{k\geq 1}$ be a sequence of vectors in $\mathbb{R}^{n}$, $\{a_k\}_{k\geq 1}$ be a sequence in $[0,1)$, and $C\geq  0$. Assume  that
	\[\left\|h_{k+1}+\sum_{i=1}^{k} a_{i} h_{i}\right\|\leq C, \quad\forall k\geq 1.\]
	 Then, $\sup_{k\geq 0}\|h_k\|\leq \|h_1\|+2C$.
\end{lemma}

\section{Convergence rate analysis}
In this section, we will demonstrate that the proposed algorithm exhibits a non-ergodic convergence rate of $\mathcal{O}(1/k^2)$.

To prove the non-ergodic convergence rate of Algorithm \ref{al1}, we begin by introducing the energy sequence  $\{\mathcal{E}_k(x,y)\}_{k\geq 1}$. Let $\{(x_k,y_k,u_k,v_k)\}_{k \geq 1}$ be the sequence generated by Algorithm \ref{al1}.  For any $(x,y)\in\mathbb{R}^n\times\mathbb{R}^m$, $\mathcal{E}_k(x,y)$ is defined  as \begin{eqnarray}\label{energy_1}
	\mathcal{E}_k(x,y) =I_k^1 +I_k^2+I_k^3+I_k^4
\end{eqnarray}
with 
\begin{numcases}{}
	 I_k^1 = t_k^2(\mathcal{L}(x_k, y)-\mathcal{L}(x,y_k)),\nonumber \\
	 I_k^2 = \frac{1}{2\alpha}\|u_k-x\|^2,\nonumber \\
	 I_k^3 = \frac{t_{k+1}^2}{2\beta}\|v_k-y\|^2,\nonumber  \\
	 I_k^4 = -t_k\langle K(u_k-x),v_k-v_{k-1}\rangle+\frac{t_k^2-\beta L_{g_2}}{2\beta}\|v_k-v_{k-1}\|^2. \nonumber 
\end{numcases}
Suppose that Assumption \ref{assA} holds. Let $\{(x_k,y_k,$ $u_k,v_k)\}_{k \geq 1}$ be the sequence generated by Algorithm \ref{al1}. Now, let us estimate $I_{k+1}^1-I_k^1$ -- $I_{k+1}^4-I_k^4$.

{\bf Estimate $I_{k+1}^1-I_k^1$}: From Algorithm \ref{al1}, we can easily get
\begin{equation}\label{eq_hxa_3}
	(x_{k+1},y_{k+1}) = \frac{t_{k+1}-1}{t_{k+1}}(x_{k},y_k) + \frac{1}{t_{k+1}}(u_{k+1},v_{k+1}).
\end{equation}
  Since $f_1$ and $g_1$ are  convex, we have
  \begin{eqnarray}\label{eq_hxa_1}
	 f_1(x_{k+1})-f_1(x) +\langle K(x_{k+1}-x),y \rangle &\leq &  \frac{t_{k+1}-1}{t_{k+1}}(f_1(x_{k})-f_1(x)+\langle K(x_{k}-x),y \rangle)\nonumber\\
	&&  +  \frac{1}{t_{k+1}} (f_1(u_{k+1})-f_1(x) +\langle K(u_{k+1}-x),y \rangle)
\end{eqnarray}
and
\begin{eqnarray}\label{eq_hxa_2}
	 g_1(y_{k+1})-g_1(y)-\langle Kx,y_{k+1}-y\rangle &\leq & \frac{t_{k+1}-1}{t_{k+1}}(g_1(y_{k})-g_1(y)-\langle Kx,y_{k}-y\rangle)\nonumber \\
	&&  +  \frac{1}{t_{k+1}} (g_1(v_{k+1})-g_1(y)-\langle Kx,v_{k+1}-y\rangle).
\end{eqnarray}
Since
\begin{eqnarray*}
&&	I_{k+1}^1-  I_k^1 =  t_{k+1}^2(\mathcal{L}(x_{k+1},y)-\mathcal{L}(x,y_{k+1}))-t_{k}^2(\mathcal{L}(x_k, y)-\mathcal{L}(x,y_k))\\
&&\quad =  t_{k+1}^2(f(x_{k+1})+g(y_{k+1})-f(x)-g(y)+\langle K(x_{k+1}-x),y\rangle-\langle Kx,y_{k+1}-y\rangle)\\
&&\qquad - t_{k+1}(t_{k+1}-1)(f(x_{k})+g(y_{k})-f(x)-g(y)+\langle K(x_{k}-x),y\rangle-\langle Kx,y_{k}-y\rangle)\\
&&\qquad + (t_{k+1}(t_{k+1}-1)-t_k^2)(\mathcal{L}(x_k, y)-\mathcal{L}(x,y_k)).
\end{eqnarray*}
 This together with \eqref{eq_hxa_3}- \eqref{eq_hxa_2} implies the  following two estimate:
\begin{eqnarray}\label{eq_es1A}
&&	I_{k+1}^1-  I_k^1\leq t_{k+1}(g_1(v_{k+1})-g_1(y)-\langle Kx,v_{k+1}-y\rangle)\nonumber \\
&&\qquad  + t_{k+1}(f(x_{k+1})+g_2(y_{k+1})-f(x)-g_2(y))\nonumber \\
&&\qquad +t_{k+1}(t_{k+1}-1)(f(x_{k+1})+g_2(y_{k+1}) - f(x_k)-g_2(y_k))\\
&&\qquad +t_{k+1}\langle K^Ty, u_{k+1}-x\rangle+ (t_{k+1}(t_{k+1}-1)-t_k^2)(\mathcal{L}(x_k, y)-\mathcal{L}(x,y_k))\nonumber
\end{eqnarray}
and 
\begin{eqnarray}\label{eq_es1B}
&&	I_{k+1}^1-  I_k^1 \leq t_{k+1}(f_1(u_{k+1})+g_1(v_{k+1})-f_1(x)-g_1(y)+\langle K(u_{k+1}-x),y\rangle-\langle Kx,v_{k+1}-y\rangle)\nonumber  \\
&&\qquad  + t_{k+1}(f_2(x_{k+1})+g_2(y_{k+1})-f_2(x)-g_2(y))\\
&&\qquad +t_{k+1}(t_{k+1}-1)(f_2(x_{k+1})+g_2(y_{k+1}) - f_2(x_k)-g_2(y_k))\nonumber\\
&&\qquad+ (t_{k+1}(t_{k+1}-1)-t_k^2)(\mathcal{L}(x_k, y)-\mathcal{L}(x,y_k)).\nonumber
 \end{eqnarray}

{\bf Estimate $I_{k+1}^2-I_k^2$}: From Algorithm \ref{al1}, we can get
\begin{eqnarray}\label{eq_he1}
	u_{k+1}-u_k &=& t_{k+1}(x_{k+1}-x_k)-(t_k-1)(x_{k}-x_{k-1})\nonumber \\
	&=& t_{k+1}\left(x_{k+1}-\left(x_k+\frac{t_k-1}{t_{k+1}}(x_{k}-x_{k-1})\right)\right)\\
	&=& t_{k+1}(x_{k+1}-\bar{x}_k)\nonumber
\end{eqnarray}
and
\begin{eqnarray*}
	 u_{k+1}-x = x_{k+1}-x + (t_{k+1}-1)(x_{k+1}-x_k). 
\end{eqnarray*}
Since $f_2$ is a convex function and  has an $L_{f_2}$-Lipschitz continuous gradient, from Lemma \ref{le_S3}  we have
\begin{eqnarray}\label{eq_he2}
	&& \langle  u_{k+1}-x,\nabla f_2(\bar{x}_k) \rangle= \langle x_{k+1}-x,\nabla f_2(\bar{x}_k) \rangle+ (t_{k+1}-1)\langle x_{k+1}-x_k,\nabla f_2(\bar{x}_k) \rangle\nonumber \\
	&&\qquad \geq f_2(x_{k+1})-f_2(x) +(t_{k+1}-1)(f_2(x_{k+1})-f_2(x_k))-\frac{L_{f_2}t_{k+1}}{2 }\|x_{k+1}-\bar{x}_k\|^2\\
	&&\qquad= f_2(x_{k+1})-f_2(x) +(t_{k+1}-1)(f_2(x_{k+1})-f_2(x_k))-\frac{L_{f_2}}{2 t_{k+1}}\|u_{k+1}-u_k\|^2.\nonumber
\end{eqnarray}
As $f_1$ is convex, we can infer
\begin{eqnarray}\label{eq_he3}
	&& \langle  u_{k+1}-x, \widetilde{\nabla} f_1(x_{k+1}) \rangle= \langle x_{k+1}-x,\widetilde{\nabla} f_1(x_{k+1}) \rangle+ (t_{k+1}-1)\langle x_{k+1}-x_k,\widetilde{\nabla} f_1(x_{k+1}) \rangle\nonumber\\
	&&\quad\geq f_1(x_{k+1})-f_1(x) +(t_{k+1}-1)(f_1(x_{k+1})-f_1(x_k))
\end{eqnarray}
and 
\begin{eqnarray}\label{eq_heA3}
	&& \langle  u_{k+1}-x, \widetilde{\nabla} f_1(u_{k+1}) \rangle\geq f_1(u_{k+1})-f_1(x).
\end{eqnarray}

If update $(x_{k+1},u_{k+1})$ by {\bf Option 1},  using the optimality condition, we get
 \begin{eqnarray*}
	x_{k+1}-\bar{x}_k=-\alpha\left(\widetilde{\nabla} f_1(x_{k+1})+\nabla f_2(\bar{x}_k) + K^T\left({v}_k+\frac{t_k}{t_{k+1}}(v_k-v_{k-1})\right)\right).
\end{eqnarray*}
Then, it follows from \eqref{eq_know} and \eqref{eq_he1} that 
\begin{eqnarray}\label{eq_es2A}
	I_{k+1}^2-I_k^2 &=& \frac{1}{2\alpha}\|u_{k+1}-x\|^2-\frac{1}{2\alpha}\|u_k-x\|^2\nonumber\\
	&=& \frac{1}{\alpha} \langle  u_{k+1}-x, u_{k+1}-u_k\rangle-\frac{1}{2\alpha}\|u_{k+1}-u_k\|^2\nonumber\\
	&=&-t_{k+1}\left\langle u_{k+1}-x, \widetilde{\nabla} f_1(x_{k+1})+\nabla f_2(\bar{x}_k) + K^T\left({v}_k+\frac{t_k}{t_{k+1}}(v_k-v_{k-1})\right)\right\rangle\nonumber\\
	&&-\frac{1}{2\alpha}\|u_{k+1}-u_k\|^2\\
	&\leq & -t_{k+1}(f(x_{k+1})-f(x) + (t_{k+1}-1)(f(x_{k+1})-f(x_k)))\nonumber\\
	&&-\frac{1-\alpha L_{f_2}}{2\alpha}\|u_{k+1}-u_k\|^2- t_{k+1}\langle u_{k+1}-x,K^Ty\rangle\nonumber\\
	&&-t_{k+1}\langle u_{k+1}-x,K^T({v}_{k}-y)\rangle-t_k\langle u_{k+1}-x, K^T(v_{k}-v_{k-1})\rangle,\nonumber
\end{eqnarray}
where the last inequality follows from \eqref{eq_he2} and \eqref{eq_he3}.

If update $(x_{k+1},u_{k+1})$ by {\bf Option 2},  using the optimality condition, we have
 \begin{eqnarray*}
	u_{k+1}-u_k=-\alpha t_{k+1}\left(\widetilde{\nabla} f_1(u_{k+1})+\nabla f_2(\bar{x}_k) + K^T\left({v}_k+\frac{t_k}{t_{k+1}}(v_k-v_{k-1})\right)\right).
\end{eqnarray*}

This,  in conjunction with \eqref{eq_know}, \eqref{eq_he2} and \eqref{eq_heA3}, implies
\begin{eqnarray}\label{eq_es2B}
	I_{k+1}^2-I_k^2 &=& \frac{1}{\alpha} \langle  u_{k+1}-x, u_{k+1}-u_k\rangle-\frac{1}{2\alpha}\|u_{k+1}-u_k\|^2\nonumber \\
	&=&-t_{k+1}\left\langle u_{k+1}-x, \widetilde{\nabla} f_1(u_{k+1})+\nabla f_2(\bar{x}_k) + K^T\left({v}_k+\frac{t_k}{t_{k+1}}(v_k-v_{k-1})\right)\right\rangle\nonumber \\
	&&-\frac{1}{2\alpha}\|u_{k+1}-u_k\|^2\\
&\leq & -t_{k+1}(f_1(u_{k+1})-f_1(x) +\langle u_{k+1}-x,K^Ty\rangle)-\frac{1-\alpha L_{f_2}}{2\alpha}\|u_{k+1}-u_k\|^2\nonumber \\
	&&-t_{k+1}(f_2(x_{k+1})-f_2(x) + (t_{k+1}-1)( f_2(x_{k+1})-f_2(x_k)))\nonumber \\
	&&-t_{k+1}\langle u_{k+1}-x,K^T({v}_{k}-y)\rangle-t_k\langle u_{k+1}-x, K^T(v_{k}-v_{k-1})\rangle.\nonumber 
\end{eqnarray}

{\bf Estimate $I_{k+1}^3-I_k^3$}:  From Algorithm \ref{al1}, we can get
\begin{eqnarray}\label{eq_heA5}
	v_{k+1}-v_k = t_{k+1}(y_{k+1}-\bar{y}_k)\nonumber
\end{eqnarray}
and
\begin{eqnarray*}
	 v_{k+1}-y = y_{k+1}-y + (t_{k+1}-1)(y_{k+1}-y_k). 
\end{eqnarray*}
 Similar to the estimate of \eqref{eq_he2}, we get
\begin{eqnarray}\label{eq_he4}
	&& \langle  v_{k+1}-y,\nabla g_2(\bar{y}_k) \rangle \nonumber \\
	&&\quad\geq  g_2(y_{k+1})-g_2(y) + (t_{k+1}-1)(g_2(y_{k+1})-g_2(y_k))-\frac{L_{g_2}}{2t_{k+1}}\|v_{k+1}-v_k\|^2.
\end{eqnarray}
By using the optimality condition, we can obtain
\[v_{k+1}-v_k = -\frac{\beta}{t_{k+1}}(\widetilde{\nabla} g_1({v_{k+1}})+g_2({\bar{y}_{k}})-Ku_{k+1}).\] 
Then from \eqref{eq_know} and \eqref{eq_he4}, we have
\begin{eqnarray}\label{eq_es3}
	I_{k+1}^3-I_k^3 &=& \frac{t_{k+2}^2-t_{k+1}^2}{2\beta}\|v_{k+1}-y\|^2 + \frac{t_{k+1}^2}{2\beta}(\|v_{k+1}-y\|^2- \|v_k-y\|^2)\nonumber \\
	&=& \frac{t_{k+2}^2-t_{k+1}^2}{2\beta}|v_{k+1}-y\|^2-\frac{t_{k+1}^2}{2\beta}\|v_{k+1}-v_k\|^2\nonumber \\
	&&-t_{k+1}\langle v_{k+1}-y,  \widetilde{\nabla} g_1(v_{k+1})+\nabla g_2(\bar{y}_k)- K{u}_{k+1}\rangle\\
	&\leq & -t_{k+1}(g_1(v_{k+1})-g_1(y) -\langle v_{k+1}-y,Kx\rangle)-\frac{t_{k+1}^2-\beta L_{g_2}}{2\beta}\|v_{k+1}-v_k\|^2\nonumber \\
	&&-t_{k+1}(g_2(y_{k+1})-g_2(y) + (t_{k+1}-1)(g_2(y_{k+1})-g_2(y_k)))\nonumber \\
	&&+\frac{t_{k+2}^2-t_{k+1}^2-\mu_g\beta t_{k+1}}{2\beta}\|v_{k+1}-y\|^2+t_{k+1}\langle v_{k+1}-y,K({u}_{k+1}-x)\rangle,\nonumber 
\end{eqnarray}
where the last inequality follows from the strong convexity of $g_1$.

{\bf Estimate $I_{k+1}^4-I_k^4$}: By direct calculation, we get 
 \begin{eqnarray}\label{eq_es4}
	I_{k+1}^4 -I_k^4 &= & -t_{k+1}\langle K(u_{k+1}-x),v_{k+1}-v_{k}\rangle +t_k\langle K(u_k-x),v_k-v_{k-1}\rangle\\
	&&+ \frac{t_{k+1}^2-\beta L_{g_2}}{2\beta}\|v_{k+1}-v_{k}\|^2- \frac{t_{k}^2-\beta L_{g_2}}{2\beta}\|v_k-v_{k-1}\|^2.\nonumber 
\end{eqnarray}

Now, we investigate the properties of  $\mathcal{E}_k(x,y)$.

\begin{lemma}\label{le_A1}
Suppose that Assumption \ref{assA} holds. Let $\{(x_k,y_k,$ $u_k,v_k)\}_{k \geq 1}$ be the sequence generated by Algorithm \ref{al1} under the following parameter assumptions: 
\begin{equation}\label{para_alp1}
	\begin{cases}
	&	\alpha \beta \|K\|^2< (1-\alpha L_{f_2})\left(1-\frac{\beta L_{g_2}}{t_1^2}\right),  \\
	 & \alpha< \frac{1}{L_{f_2}},\quad \beta<\frac{t_1^2}{L_{g_2}}.
	\end{cases}
\end{equation}
Let the energy sequence $\{\mathcal{E}_k(x,y)\}_{k\geq 1}$ be defined in \eqref{energy_1}. Then  for any $(x,y)\in\mathbb{R}^n\times \mathbb{R}^m$, we have
\begin{eqnarray*}
 \mathcal{E}_{k+1}(x,y)- \mathcal{E}_k(x,y)\leq (t_{k+1}(t_{k+1}-1)-t_k^2)(\mathcal{L}(x_k, y)-\mathcal{L}(x,y_k))
\end{eqnarray*}and 
\[\mathcal{E}_k(x,y) \geq  t_k^2(\mathcal{L}(x_k, y)-\mathcal{L}(x,y_k))+\left(\frac{1}{2\alpha}-\frac{t_1^2\beta\|K\|^2}{2(t_1^2-{\beta L_{g_2}})}\right)\|u_k-x\|^2+\frac{t_{k+1}^2}{2\beta}\|v_k-y\|^2.\]
\end{lemma}
\begin{proof}
{\bf Case 1}: Update $(x_{k+1},u_{k+1})$ by {\bf Option 1}.

Combining \eqref{eq_es1A}, \eqref{eq_es2A}, \eqref{eq_es3} and \eqref{eq_es4}, we can get
\begin{eqnarray}\label{eq_he5}
	&& \mathcal{E}_{k+1}(x,y)-\mathcal{E}_k(x,y) =  I_{k+1}^1 -I_k^1+I_{k+1}^2 -I_k^2+I_{k+1}^3 -I_k^3+I_{k+1}^4 -I_k^4\nonumber\\
	&&\qquad\quad\leq (t_{k+1}(t_{k+1}-1)-t_k^2)(\mathcal{L}(x_k, y)-\mathcal{L}(x,y_k))-\frac{1-\alpha L_{f_2}}{2\alpha}\|u_{k+1}-u_k\|^2\\
	&&\qquad\qquad-\frac{t^2_{k}-\beta L_{g_2}}{2\beta}\|v_{k}-v_{k-1}\|^2-t_k\langle u_{k+1}-u_k, K^T(v_{k}-v_{k-1})\nonumber \\
	&&\qquad\qquad +\frac{t_{k+2}^2-t_{k+1}^2-\mu_g\beta t_{k+1}}{2\beta}\|v_{k+1}-y\|^2.\nonumber
\end{eqnarray}
From Algorithm \ref{al1}, we can get  
\[ t_{k+1}^2\leq t_k^2+\beta \mu_g t_k.\]
Then, from \eqref{eq_he5}, we have
\begin{eqnarray}\label{eq_he6}
	\mathcal{E}_{k+1}(x,y)-\mathcal{E}_k(x,y) &\leq & (t_{k+1}(t_{k+1}-1)-t_k^2)(\mathcal{L}(x_k, y)-\mathcal{L}(x,y_k))\nonumber \\
	&&+t_k\langle u_k-u_{k+1}, K^T(v_{k}-v_{k-1})\rangle-\frac{1-\alpha L_{f_2}}{2\alpha}\|u_{k+1}-u_k\|^2\nonumber \\
	&&-\frac{t_k^2-\beta L_{g_2}}{2\beta}\|v_{k}-v_{k-1}\|^2.
\end{eqnarray}
As we have
 \[ \alpha \beta \|K\|^2< (1-\alpha L_{f_2})\left(1-\frac{\beta L_{g_2}}{t_1^2}\right),  \]
we can select
\begin{equation}\label{chosR}
	r\in\left(\frac{\beta L_{g_2}}{t_1^2},1-\frac{\alpha\beta\|K\|^2}{1-\alpha L_{f_2}}\right).
\end{equation}
It follows from \eqref{para_alp1} that $0<r<1$. Then we can apply \eqref{eq_know1} to obtain
\begin{eqnarray*}
	 t_k\langle u_k-u_{k+1}, K^T(v_{k}-v_{k-1})\rangle \leq \frac{1}{2}\left(\frac{\beta\|K\|^2}{(1-r)}\| u_{k+1}-u_k\|^2+\frac{(1-r)t_k^2}{\beta}\|v_{k}-v_{k-1}\|^2\right).
\end{eqnarray*}
This, in combination with \eqref{eq_he6}, implies
\begin{eqnarray*}
&&\mathcal{E}_{k+1}(x,y)-\mathcal{E}_k(x,y)  \\
&&\qquad \leq  (t_{k+1}(t_{k+1}-1)-t_k^2)(\mathcal{L}(x_k, y)-\mathcal{L}(x,y_k))\\
&&\qquad\quad  -\left(\frac{1-\alpha L_{f_2}}{2\alpha}-\frac{\beta\|K\|^2}{2(1-r)}\right)\|u_{k+1}-u_k\|^2-\frac{rt_k^2-\beta L_{g_2}}{2\beta}\|v_{k}-v_{k-1}\|^2\\
&&\qquad \leq (t_{k+1}(t_{k+1}-1)-t_k^2)(\mathcal{L}(x_k, y)-\mathcal{L}(x,y_k)),
\end{eqnarray*}
where the last inequality follows from \eqref{chosR} and $t_k\geq t_1$ for all $k\geq 1$.

Using \eqref{eq_know1} again, we can derive
\[t_k\langle K(u_k-x),v_k-v_{k-1}\rangle \leq \frac{1}{2}\left(\frac{\beta\|K\|^2}{1-\frac{\beta L_{g_2}}{t_1^2}}\| u_{k}-x^*\|^2+\frac{t_k^2(1-\frac{\beta L_{g_2}}{t_1^2})}{\beta}\|v_{k}-v_{k-1}\|^2\right).\]
Moreover, since $\frac{\beta L_{g_2} t_k^2}{t_1^2}\geq  \beta L_{g_2}$,  we have
\begin{eqnarray*}
	I_{k}^2+I_{k}^4\geq \left(\frac{1}{2\alpha}-\frac{t_1^2\beta\|K\|^2}{2(t_1^2-{\beta L_{g_2}})}\right)\|u_k-x\|^2.
\end{eqnarray*}
Finally, from \eqref{energy_1}, we can deduce the results.

{\bf Case 2}: Update $(x_{k+1},u_{k+1})$ by {\bf Option 2}.

Combining \eqref{eq_es1B}, \eqref{eq_es2B}, \eqref{eq_es3} and \eqref{eq_es4}, we also can get the following inequality:
\begin{eqnarray*}
	&& \mathcal{E}_{k+1}(x,y)-\mathcal{E}_k(x,y) =  I_{k+1}^1 -I_k^1+I_{k+1}^2 -I_k^2+I_{k+1}^3 -I_k^3+I_{k+1}^4 -I_k^4\nonumber\\
	&&\qquad\quad\leq (t_{k+1}(t_{k+1}-1)-t_k^2)(\mathcal{L}(x_k, y)-\mathcal{L}(x,y_k))-\frac{1-\alpha L_{f_2}}{2\alpha}\|u_{k+1}-u_k\|^2\\
	&&\qquad\qquad-\frac{t^2_{k}-\beta L_{g_2}}{2\beta}\|v_{k}-v_{k-1}\|^2-t_k\langle u_{k+1}-u_k, K^T(v_{k}-v_{k-1})\nonumber \\
	&&\qquad\qquad +\frac{t_{k+2}^2-t_{k+1}^2-\mu_g\beta t_{k+1}}{2\beta}\|v_{k+1}-y\|^2.\nonumber
\end{eqnarray*}
Then, by similar arguments as Case 1, we also can get results.
\end{proof}

With Lemma \ref{le_A1} in hand, we prove the non-ergodic $\mathcal{O}(1/k^2)$ convergence rate of Algorithm \ref{al1}.

\begin{theorem}\label{th_A1}
Suppose that Assumption \ref{assA} holds. Let $\{(x_k,y_k,u_k,v_k)\}_{k \geq 1}$ be the sequence generated by Algorithm \ref{al1} with  parameters satisfying \eqref{para_alp1}, and  $(x^*,y^*)\in\Omega$  ($y^*$ is unique). Then the following  conclusions hold:
\begin{itemize}
  \item [(a).] The gap function satisfies
\[\mathcal{L}(x_k, y^*)-\mathcal{L}(x^*,y_k) \leq \frac{\mathcal{E}_1(x^*,y^*)}{t_k^2},  \]
  moreover, when $\beta\mu_g>1+1/t_1$,  for any compact set  $\mathcal{B}_1\times \mathcal{B}_2 \subset \mathbb{R}^n\times\mathbb{R}^m$ containing a saddle point:
\[\mathcal{G}_{\mathcal{B}_1\times \mathcal{B}_2}(x_k,y_k)  \leq \frac{D(B_1,B_2)}{t_{k}^2}, \]
where $D(B_1,B_2) = \max_{x\in B_1,y\in B_2} \mathcal{E}_1(x,y)$ and 
\[\mathcal{E}_1(x,y)= t_1^2(\mathcal{L}(x_1, y)-\mathcal{L}(x,y_1))+\frac{1}{2\alpha}\|x_1-x\|^2+\frac{t_2^2}{2\beta}\|y_1-y\|^2.\]
\item [(b).] The sequence $\{(x_k,y_k,u_k,v_k)\}_{k \geq 1}$ is bounded, and
\[\|y_k-y^*\|^2\leq \frac{2\mathcal{E}_1(x^*,y^*)}{\mu_g t_k^2} , \quad \|v_k-y^*\|^2\leq \frac{2\beta\mathcal{E}_1(x^*,y^*)}{t_{k+1}^2}, \]
\[\|x_{k}-x_{k-1}\|\leq \mathcal{O}\left(\frac{1}{t_k}\right), \quad \|y_{k}-y_{k-1}\|\leq \mathcal{O}\left(\frac{1}{t_k^2}\right).\]
\item [(c).] When $\beta\mu_g>1+1/t_1$: every limit point of $\{(x_k,y_k)\}_{k\ge 1}$ is a saddle point of problem \eqref{ques}.
\end{itemize}
\end{theorem}
\begin{proof}
$(a)$. Fix $(x,y)=(x^*,y^*)$. It follows from \eqref{saddle} and Lemma \ref{le_A1} that
\begin{eqnarray}\label{eq_4101}
 \mathcal{E}_{k+1}(x^*,y^*)- \mathcal{E}_k(x^*,y^*)\leq (t_{k+1}(t_{k+1}-1)-t_k^2)(\mathcal{L}(x_k, y^*)-\mathcal{L}(x^*,y_k))\leq 0,
\end{eqnarray}
where the last equality follows from $t_{k+1}\leq \frac{1+\sqrt{1+4t_k^2}}{2}$ which implies $t_{k+1}^2-t_{k+1}\leq t_k^2$. From \eqref{para_alp1}, we have $\frac{1}{2\alpha}\geq\frac{t_1^2\beta\|K\|^2}{2(t_1^2-\beta L_{g_2})}$, then by \eqref{eq_4101} and Lemma \ref{le_A1} we get 
	\begin{eqnarray*}
		 t_k^2(\mathcal{L}(x_k, y^*)-\mathcal{L}(x^*,y_k))\leq \mathcal{E}_k(x^*,y^*) \leq \mathcal{E}_1(x^*,y^*).
	\end{eqnarray*}
This implies 
\[\mathcal{L}(x_k, y^*)-\mathcal{L}(x^*,y_k) \leq \frac{\mathcal{E}_1(x^*,y^*)}{t_k^2}.\]

Given that  $t_{k+1}^2\leq  t^2_k + t_{k+1}$, we have $(t_{k+1}-\frac{1}{2})^2= t^2_k+\frac{1}{4}\leq (t_{k}+\frac{1}{2})^2$, and therefore $t_{k+1}\leq t_k+1$. With the assumption that  $\beta\mu_g> 1+\frac{1}{t_1}$ and $t_{k}\geq t_1$, we can derive
\[\beta\mu_g t_{k} > t_{k}+\frac{t_{k}}{t_1}\geq t_{k}+1\geq t_{k+1}.\]
Then if $t_{k+1}^2=  t^2_k + t_{k+1}$, we have $t_{k+1}\leq \sqrt{t_k^2+\mu_g\beta t_k}$, such that  $t_{k+1} = \min\left\{\frac{1+\sqrt{1+4t_k^2}}{2}, \sqrt{t_k^2+\mu_g\beta t_k}\right\}$ can reduce to $t_{k+1}^2=  t^2_k + t_{k+1}$. This together with Lemma \ref{le_A1} implies 
\[\mathcal{E}_{k+1}(x,y)\leq \mathcal{E}_k(x,y)\]
for any $(x,y)\in\mathbb{R}^m\times \mathbb{R}^n$ and $k\geq 1$. Since $\frac{1}{2\alpha}\geq\frac{t_1^2\beta\|K\|^2}{2(t_1^2-\beta L_{g_2})}$, we have 
\begin{equation}\label{eq_he7}
	\mathcal{L}(x_k, y)-\mathcal{L}(x,y_k) \leq \frac{\mathcal{E}_1(x,y)}{t_k^2}
\end{equation}
for any $(x,y)\in\mathbb{R}^n\times \mathbb{R}^m$ and $k\geq 1$. This, along with  the definition of $\mathcal{G}_{\mathcal{B}_1\times \mathcal{B}_2}(x_k,y_k)$ gives us $(a)$.

$(b)$. Take $(x,y) = (x^*,y^*)$. From Lemma \ref{le_A1} and \eqref{saddle}, we have
 \[\left(\frac{1}{2\alpha}-\frac{t_1^2\beta\|K\|^2}{2(t_1^2-{\beta L_{g_2}})}\right)\|u_k-x^*\|^2+\frac{t_{k+1}^2}{2\beta}\|v_k-y^*\|^2\leq \mathcal{E}_1(x^*,y^*).\]
Since $\frac{1}{2\alpha}-\frac{t_1^2\beta\|K\|^2}{2(t_1^2-{\beta L_{g_2}})}>0$ and $t_k\geq 1$, it follows that $\{(u_k,v_k)\}_{k\ge 1}$ is bounded, and 
\begin{equation}\label{eq_he9}
	\|v_k-y^*\|^2\leq \frac{2\beta\mathcal{E}_1(x^*,y^*) }{t_{k+1}^2}. 
\end{equation}
Take $(x,y) = (x^*,y^*)$ in \eqref{eq_he7}. It follows from \eqref{saddle_AA} and the fact that  $g_1$ is $\mu_g$-strongly convex that
\begin{equation}\label{eq_he10}
	\|y_k-y^*\|^2\leq \frac{2\mathcal{E}_1(x^*,y^*) }{\mu_g t_k^2}.
\end{equation}
Then $\{y_k\}_{k\geq 1}$ is bounded. From  Algorithm \ref{al1}, we can deduce
\[ u_{k+1} - u_k= x_{k}-x_{k-1} +t_{k+1}(x_{k+1}-x_k)-t_k(x_k-x_{k-1})\]
Considering that $x_1=x_0$, we obtain
\begin{eqnarray*}
	 u_{k+1}  -  u_{1}&=& \sum^{k}_{i=1} (u_{i+1} - u_i)= t_{k+1}(x_{k+1}-x_k) + \sum^{k}_{i=1} (x_{i}-x_{i-1}).
\end{eqnarray*}
Given that $\{u_k\}_{k\ge 1}$ is bounded, there exists $C>0$ such that
\[\left\|t_{k+1}(x_{k+1}-x_k) + \sum^{k}_{i=1} (x_{i}-x_{i-1})\right\|\leq C.  \]
Utilizing Lemma \ref{le_S4} with $h_k = t_{k}(x_{k}-x_{k-1})$ and $a_k=1/t_{k}$ yields
\[\sup_{k\geq 1} t_{k}\|x_{k}-x_{k-1}\| =\|h_k\|\leq 2C,\]
and, in turn
\[\|x_{k}-x_{k-1}\|\leq \frac{2C}{t_{k}}=\mathcal{O}\left(\frac{1}{t_k}\right).\]
Together with Algorithm \ref{al1} and the boundedness of $\{u_k\}_{ k\ge 1}$, this implies
\begin{eqnarray*}
	\sup_{k\geq 1}\|x_k\| &=&\sup_{k\geq 1}\|u_{k+1}-t_{k+1}(x_{k+1}-x_k)\|\\
	&\leq &\sup_{k\geq 1}\|u_{k+1}\|+\sup_{k\geq 1}\|t_{k+1}(x_{k+1}-x_k)\|\\
	&<& +\infty,
\end{eqnarray*}
demonstrating that $\{x_k\}_{k\geq 1}$ is bounded. From Algorithm \ref{al1}, we have 
\begin{eqnarray*}
	\|y_{k}-y_{k-1}\| &=& \frac{1}{t_{k}}\|v_{k}-y^*-(y_{k-1}-y^*)\|\\
	&\leq &\frac{1}{t_{k}}(\|v_{k}-y^*\|+\|y_{k-1}-y^*\|)	\\
	&=&\mathcal{O}\left(\frac{1}{t_k^2}\right),
\end{eqnarray*}
where the last equality follows from \eqref{eq_he9} and \eqref{eq_he10}. This proves  $(b)$.

$(c)$. From $(b)$, it is evident that $\{(x_k,y_k)\}_{k \geq 1}$ is  bounded. Let  $(x^{\infty},y^\infty)$ be a limit point of  $\{(x_k,y_k)\}_{k \geq 1}$. Then there exists a subsequence $\{(x_{k_j},y_{k_j})\}_{j\ge 1}$ of 
 $\{(x_k,y_k)\}_{k \geq 1}$ such that
\[\lim_{j\to+\infty}(x_{k_j},y_{k_j}) =(x^{\infty},y^\infty).  \]
Passing to the limit in \eqref{eq_he7}, we obtain 
\begin{eqnarray}\label{eq_aa11}
	\mathcal{L}(x^{\infty}, y)-\mathcal{L}(x,y^{\infty}) \leq \lim_{k\to+\infty}\frac{\mathcal{E}_1(x,y)}{t_k^2}
\end{eqnarray}
for any $(x,y)\in\mathbb{R}^m\times \mathbb{R}^n$. 

Since $x_1\in  dom(f)$, $y_1\in  dom(g)$, and $dom(f), dom(g)$ are closed and convex set, it follows from Algorithm \ref{al1} that $x^{\infty}\in dom(f)$ and $y^{\infty}\in dom(g)$.  

When $x\notin dom(f)$: we can easily get 
\[\mathcal{L}(x^{\infty}, y^{\infty})\leq \mathcal{L}(x,y^{\infty}),\]
when $x\in dom(f)$: from \eqref{eq_aa11} we have  
\[\mathcal{L}(x^{\infty}, y^{\infty})-\mathcal{L}(x,y^{\infty}) \leq \lim_{k\to+\infty}\frac{\mathcal{E}_1(x,y^{\infty})}{t_k^2} = 0.
\]
Similarly, when $y\notin dom(g)$:
\[ \mathcal{L}(x^{\infty},y) \leq \mathcal{L}(x^{\infty}, y^{\infty}),\]
when $y\in dom(g)$: from \eqref{eq_aa11} we have  
\[\mathcal{L}(x^{\infty},y) \leq \mathcal{L}(x^{\infty}, y^{\infty}).\]
In summary, for any $(x,y)\in\mathbb{R}^m\times \mathbb{R}^n$, we have
\[  \mathcal{L}(x^{\infty},y) \leq \mathcal{L}(x^{\infty}, y^{\infty})\leq  \mathcal{L}(x,y^{\infty}).  \]
This demonstrates   that $(x^{\infty},y^\infty)$ satisfies \eqref{saddle}, and therefore is a saddle point.

\end{proof}

\begin{remark}
From Lemma \ref{le_S2} and Theorem \ref{th_A1},  we get  that Algorithm \ref{al1} enjoys the non-ergodic $\mathcal{O}(1/k^2)$ convergence rate of gap function,  the $\mathcal{O}(1/k)$ convergence rate of $\|x_{k}-x_{k-1}\|$  and the $\mathcal{O}(1/k^2)$ convergence rate of $\|y_{k}-y_{k-1}\|$. Note that the strong convex assumption of $g$ allows the improved convergence rate of $\|y_{k}-y_{k-1}\|$.
\end{remark}

\begin{remark}
When $\beta\mu_g>1+1/t_1$, the parameter setting of $t_{k+1} = \min\left\{\frac{1+\sqrt{1+4t_k^2}}{2}, \sqrt{t_k^2+\mu_g\beta t_k}\right\}$ becomes $t_{k+1} = \frac{1+\sqrt{1+4t_k^2}}{2}$. In this case,  the non-ergodic $\mathcal{O}(1/k^2)$ convergence rate derived in Theorem \ref{th_A1} is also consistent with  those of  accelerated methods based on Tseng's scheme  in \cite{Xu2017,HeADMM,Tseng2008} and  accelerated methods based on Beck-Teboulle's scheme \cite{HeNA, BeckIma, Chambolle}. In contrast,  the accelerated first-order primal-dual algorithms in \cite{ChambolleP,chang2021,Fercoq,Jiang2021,Tan2020} for problem \eqref{ques} in the partially strongly convex case only achieve ergodic $\mathcal{O}(1/k^2)$ convergence rates; The accelerated first-order primal-dual algorithms in \cite{Tran2022,Zhu2022} need to multiple complex  parameter settings to achieve  non-ergodic  $\mathcal{O}(1/k^2)$ convergence rates,  a fixed proximal center $\dot{y}$ is required  in \cite{Tran2022}, and they do not consider the problem \eqref{ques} with $f$ and $g$ having a composite structure.
\end{remark}

\begin{remark}
As shown in the proof of Lemma \ref{le_A1}, we can see that when  $f_2$  vanishes in problem \eqref{ques}, i.e.,  $f=f_1$, then parameter settings \eqref{para_alp1} hold for any $	L_{f_2}>0$ and it can be simplified to  
\[\alpha>0,\quad \beta<\frac{t_1^2}{L_{g_2}},\quad	 \alpha\beta\|K\|^2< {1-\frac{\beta L_{g_2}}{t_1^2}}. \]
Similarly,  if  $g_2$  vanishes in problem \eqref{ques}, the parameter requirement in \eqref{para_alp1} for  Algorithm \ref{al1}  reduces to  
\[\alpha< \frac{1}{L_{f_2}}, \quad \beta>0, \quad \alpha \beta \|K\|^2< (1-\alpha L_{f_2}).\]
When both $f_2$ and  $g_2$  vanish in problem \eqref{ques}, the parameter condition \eqref{para_alp1} for  Algorithm \ref{al1} becomes $\alpha\beta\|K\|^2<1$ and $\alpha,\beta>0$.
\end{remark}

\section{Numerical experiments}
In this section, we present two numerical experiments aimed at validating the performance of Algorithm \ref{al1}. We apply Algorithm \ref{al1} to solve the $l_1$ regularized least squares problem and the nonnegative least squares problem. The numerical results demonstrate the effectiveness and superior performance of Algorithm \ref{al1} compared to existing accelerated algorithms. To further enhance the practical performance of the algorithms, one can employ a restarting strategy, and theoretical guarantees can be established in a similar manner as shown in \cite{TranMPC}. It's worth noting that the parameter settings of all algorithms in the experiments satisfy the parameter assumptions of the theoretical convergence rates.
All codes are implemented using Python 3.8 on a MacBook laptop equipped with  Intel Core i5 CPU running at 2.30GHz and 8 GB of memory.

\subsection{$l_1$ regularized least squares.}
We study the following $l_1$ regularized problem:
\begin{equation}\label{ex1}
	 \min_{x\in\mathbb{R}^{n}} F(x) = \lambda\|x\|_1+\frac{1}{2}\|Kx-b\|^2,
\end{equation}
where $\lambda>0$, $K\in\mathbb{R}^{m\times n}$ and $b\in\mathbb{R}^{m}$. Clearly, \eqref{ex1} can be reformulated as the following saddle point problem 
\begin{eqnarray*}
 \min_{x\in\mathbb{R}^{n}} \max_{y\in\mathbb{R}^{m}} 	\lambda\|x\|_1 +\langle Kx,y\rangle-\frac{1}{2}\|y+b\|^2.
\end{eqnarray*}
Let $f(x)=\lambda\|x\|_1$, $g(y)=\frac{1}{2}\|y+b\|^2$. It is easy to verify that $g$ is strongly convex with constant $1$ and that Assumption \ref{assA} is satisfied. Next, We compare the performance of PDA (\eqref{pda} with $\theta=1$),  accelerated primal-dual algorithm (APDA) \cite[Algorithm 2]{ChambolleP}, FISTA \eqref{FISTA}, the sparse reconstruction by separable approximation (SpaRSA) \cite{Wright2009}, and inertial accelerated  primal-dual algorithm (Algorithm \ref{al1} with {\bf Option 1} and {\bf Option 2} (for short AL1-op1, AL-op2). We set the parameters as follows:

$\bullet$ PDA: $\alpha = \frac{1}{20\|K\|}, \beta =  \frac{20}{\|K\|}$;

$\bullet$ APDA: $\alpha = \frac{1}{\|K\|}, \beta =  \frac{1}{\|K\|}$;

$\bullet$ FISTA: $\alpha = \frac{1}{\|K\|^2}$;

$\bullet$ SpaRSA: $M=5, \sigma=0.01, \alpha_{\max} = 1/\alpha_{\min} = 10^{30}$;

$\bullet$ AL1-op1, AL1-op2: $t_0=5$, $\alpha = \frac{0.98}{2\|K\|}$, $\beta = \frac{2}{\|K\|}$.

Let $\lambda=0.1$, and let $K$ be generated by a Gaussian distribution. The vector $\bar{x}$ has $0.95n$ non-zero elements, which are generated by a uniform distribution in the range $[-10, 10]$. The noise vector $\omega \in \mathbb{R}^m$ has entries drawn from  $\mathcal{N}(0,0.1)$. The observed vector $b$ is given by $b=K\bar{x}+\omega$. Figure \ref{fig1} illustrates the convergence results, plotting $F(x_k) - \min F(x^*)$ against the number of iterations and CPU time. In this experiment, we execute all algorithms for a sufficient number of iterations to obtain the ground truth solution $\min F$. As depicted in Figure \ref{fig1}, both AL1-op1 and AL1-op2 exhibit superior numerical performance compared to the other algorithms considered, as evidenced by fewer iterations  and less CPU time.

\begin{figure}[htbp]
\centering
\subfigure[$m=1000,\ n=2000$]{
\begin{minipage}[t]{0.95\linewidth}
\centering
\includegraphics[width=2.6in]{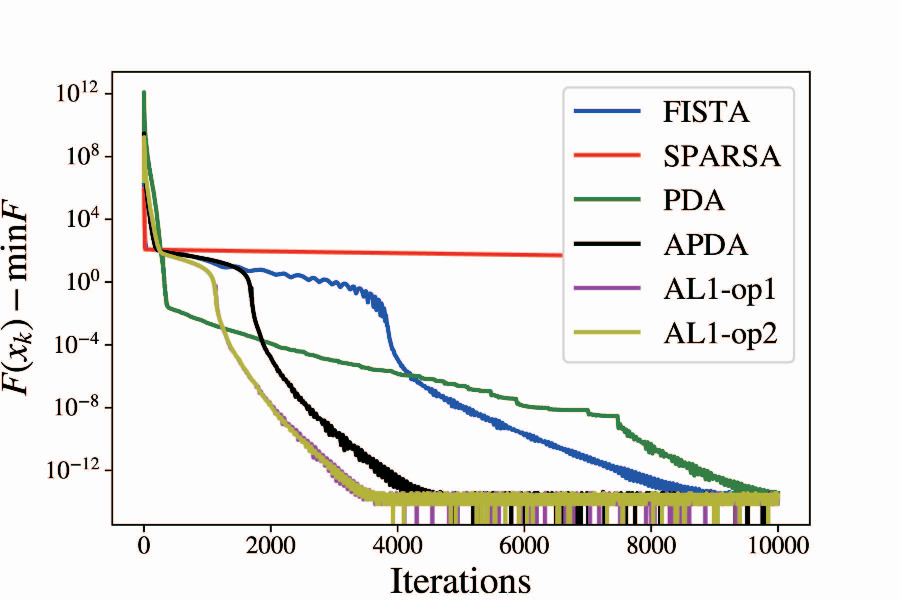}
\centering
\includegraphics[width=2.6in]{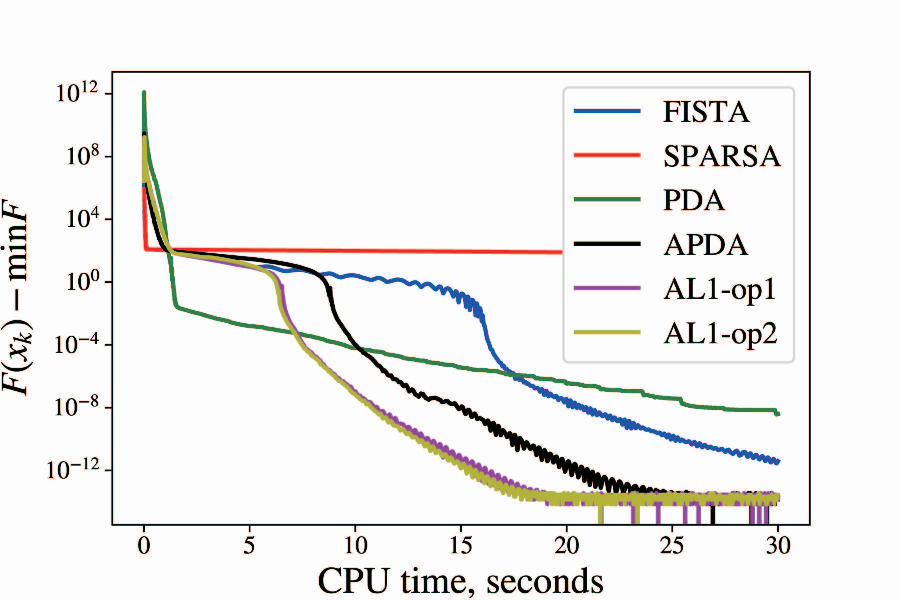}
\end{minipage}%
}%
\\
\subfigure[$m=2000,\ n=4000$]{
\begin{minipage}[t]{0.95\linewidth}
\centering
\includegraphics[width=2.6in]{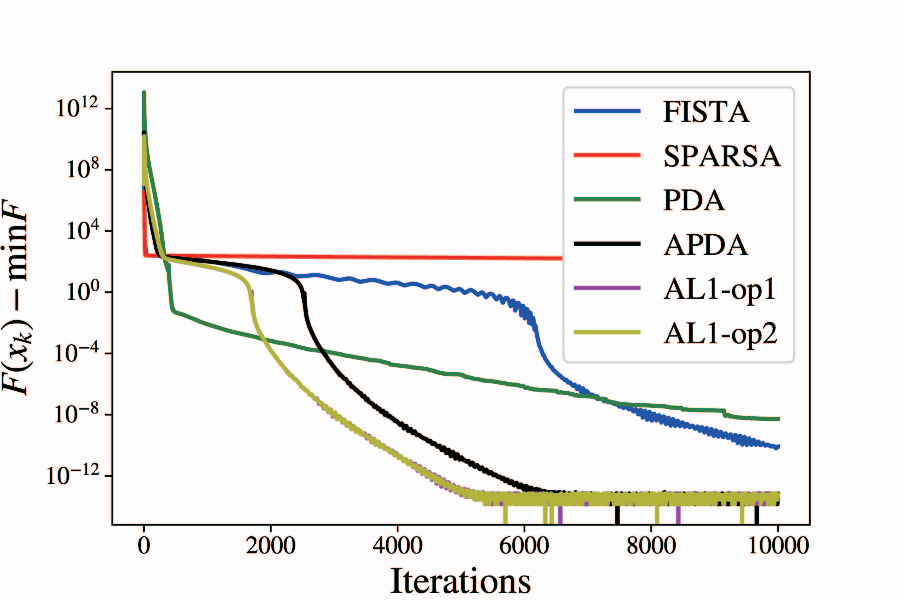}
\centering
\includegraphics[width=2.6in]{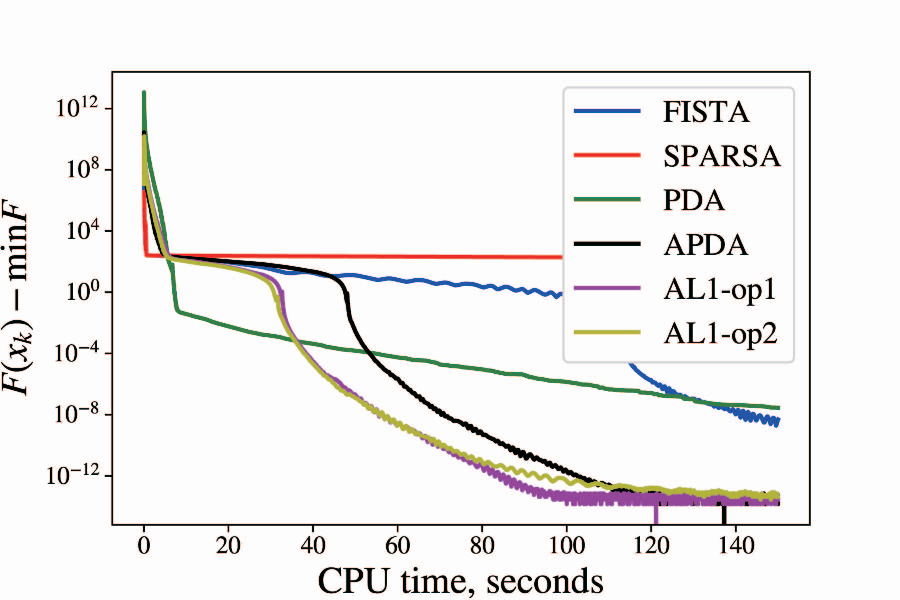}
\end{minipage}%
}%
\caption{Numerical results of algorithms for problem \eqref{ex1}}\label{fig1}
\end{figure}

\subsection{Nonnegative least squares.}

\begin{figure}[htbp]
\centering
\subfigure[$s=0.1$]{
\begin{minipage}[t]{0.95\linewidth}
\centering
\includegraphics[width=2.6in]{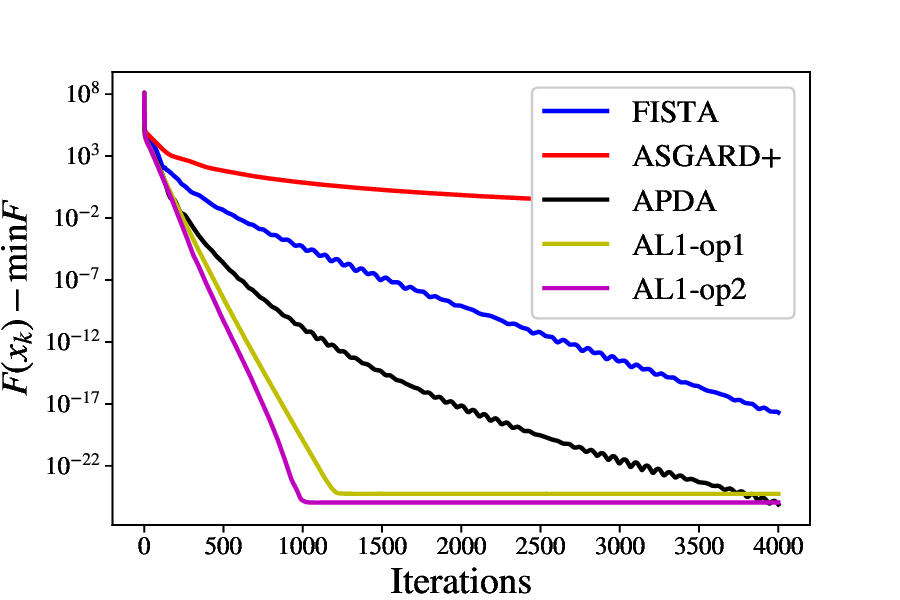}
\centering
\includegraphics[width=2.6in]{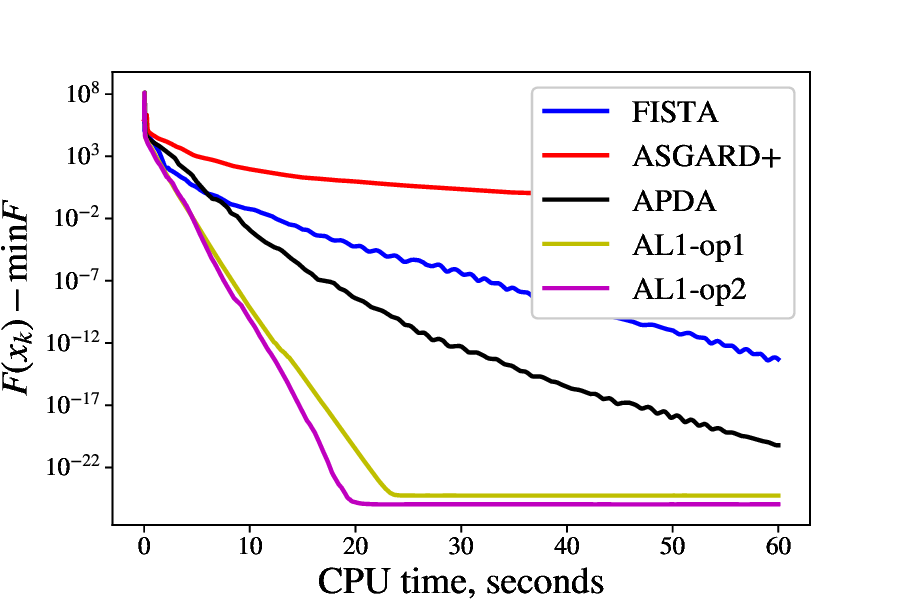}
\end{minipage}%
}%
\\
\subfigure[$s=0.5$]{
\begin{minipage}[t]{0.95\linewidth}
\centering
\includegraphics[width=2.6in]{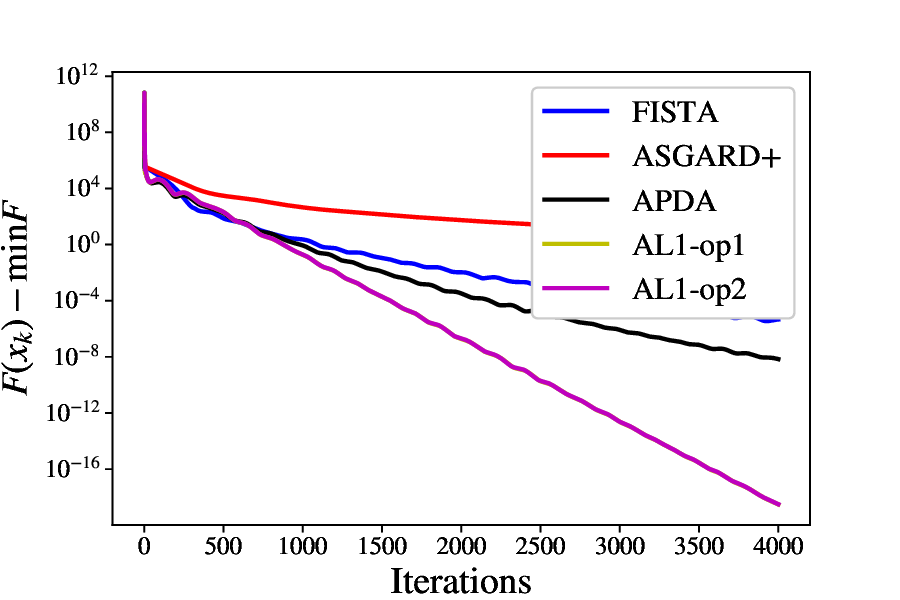}
\centering
\includegraphics[width=2.6in]{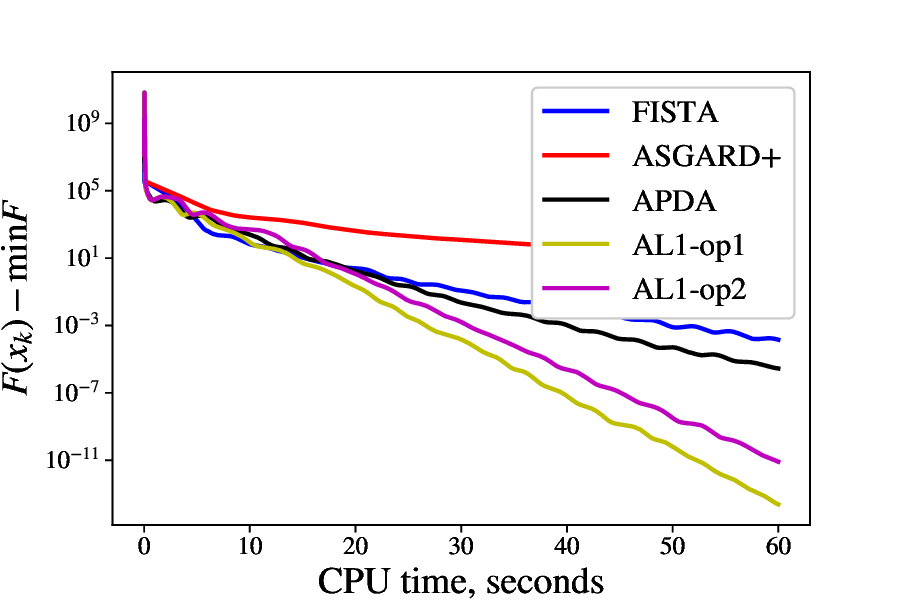}
\end{minipage}%
}%
\\
\subfigure[$s=1$]{
\begin{minipage}[t]{0.95\linewidth}
\centering
\includegraphics[width=2.6in]{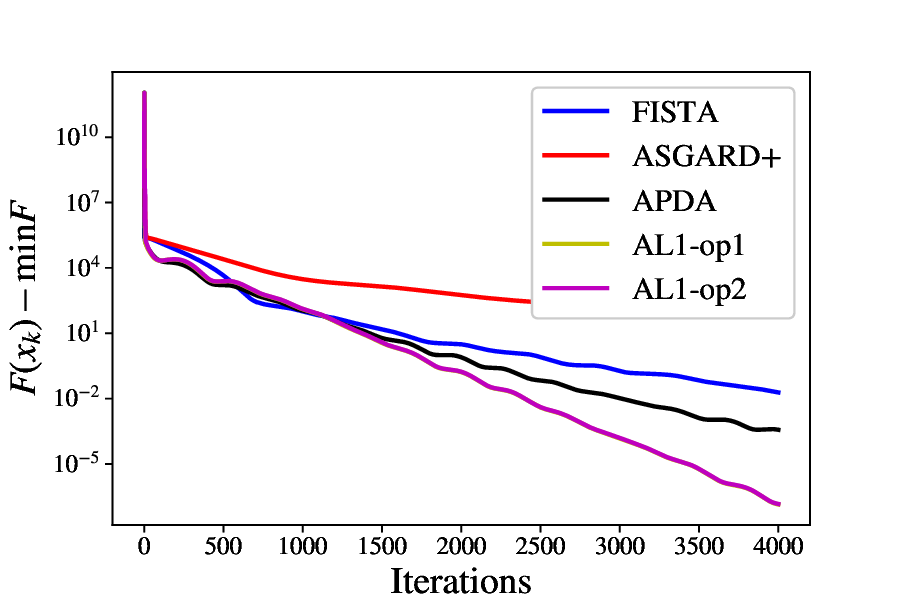}
\centering
\includegraphics[width=2.6in]{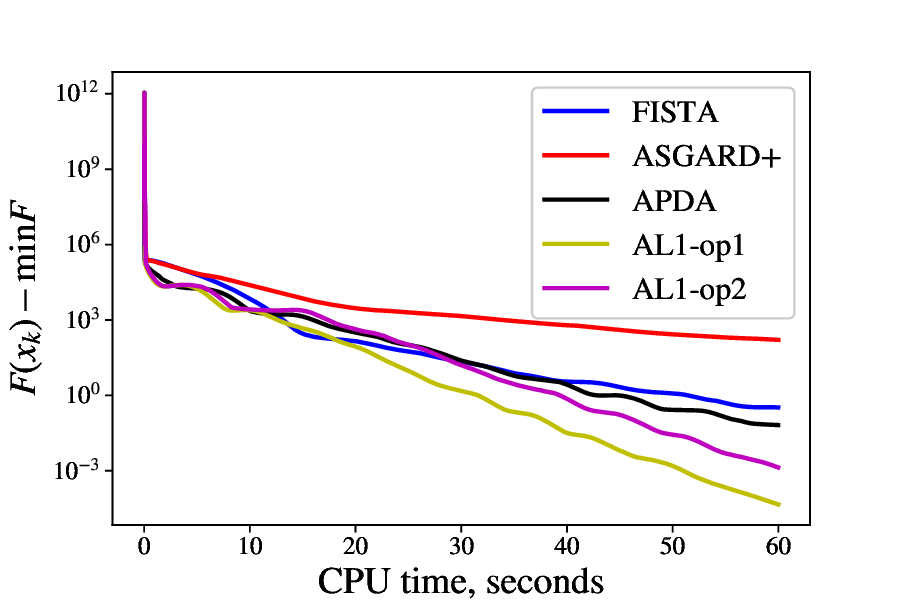}
\end{minipage}%
}%
\caption{Numerical results of algorithms for problem \eqref{ex2}}\label{fig2}
\end{figure}

In this part, we consider another regularized least squares problem:
\begin{equation}\label{ex2}
	 \min_{x\geq 0} F(x) = \frac{1}{2}\|Kx-b\|^2,
\end{equation}
where  $K\in\mathbb{R}^{m\times n}$ and $b\in\mathbb{R}^{m}$. This problem admits the following saddle point formulation:
\begin{eqnarray*}
	 \min_{x\in\mathbb{R}^{n}} \max_{y\in\mathbb{R}^{m}} \delta_{R_+^n}(x) +\langle Kx,y\rangle-\frac{1}{2}\|y+b\|^2,
\end{eqnarray*}
where $\delta_{\mathbb{R}_+^n}(x)$ is the indicate function of the set $\mathbb{R}_+^n:=\{x\in \mathbb{R}^{n}|x\ge 0\}$. 

We compare the performance of  accelerated primal-dual algorithm (APDA) \cite[Algorithm 2]{ChambolleP}, FISTA \eqref{FISTA}, accelerated smoothed gap reduction (ASGARD+) \cite{Tran2020}, and inertial accelerated  primal-dual algorithm (Algorithm \ref{al1} with Option 1 and Option 2 (AL1-op1, AL-op2). The parameters are set as follows:

$\bullet$ FISTA: $\alpha = \frac{1}{\|K\|^2}$;

$\bullet$ APDA: $\alpha = \frac{1}{\|K\|}, \beta =  \frac{1}{\|K\|}$;

$\bullet$ ASGARD+: $\beta_0 = \frac{0.382}{\|K\|^2}$, $\tau_{k+1}=\frac{\tau_k}{2}(\sqrt{\tau_k^2+4}-\tau_k)$ with $\tau_0 = 1$;

$\bullet$ AL1-op1, AL1-op2: $t_0=1.2$, $\alpha = \frac{0.98}{\|K\|}$, $\beta = \frac{1}{\|K\|}$.

We generate a random matrix $K\in\mathbb{R}^{m\times n}$ with density $s \in (0,1]$. The nonzero entries of $K$ are independently generated from a uniform distribution in $[0, 0.1]$.

 We generate $\bar{x}\in\mathbb{R^n_{+}}$ as a sparse vector with $0.05n$ nonzero entries, drawn from a uniform distribution in $[0, 100]$, and set $b=K\bar{x}$. Here, we consider $m=4000$ and $n=2000$. Figure  \ref{fig2} illustrates the convergence results with $F(x_k)-\min F$ versus iterations and CPU seconds. The efficiency of Algorithm \ref{al1} is evident in Figure \ref{fig2}, particularly in the case of sparse real matrices $A$.

\section{Conclusions}
In this paper, we propose an  inertial accelerated primal-dual algorithm  incorporating Nesterov's  extrapolation to  solve the saddle point problem \eqref{ques}. By constructing   energy sequences, we demonstrate that  the proposed algorithm  achieves a non-ergodic  $\mathcal{O}(1/k^2)$ rate under  the partially strong convexity assumption. We compare Algorithm \ref{al1} with some existing methods by testing  the $l_1$ regularized least squares problem and the nonnegative least squares problem. Our numerical experiments validate the effectiveness and superior performance of our approaches compared to existing methods.

\appendix

\section*{CRediT authorship contribution statement}
{\bf Xin He:} Conceptualization, Software, Visualization, Writing – original draft.
{\bf Nan-Jing Huang:}  Conceptualization, Investigation,  Funding acquisition, Writing – review \& editing. 
{\bf Ya-Ping Fang} Methodology, Supervision, Writing – review \& editing.

\section*{Declaration of competing interest}

The authors declare that they have no known competing financial interests or personal relationships that could have appeared
to influence the work reported in this paper.

\section*{Data availability}
No data was used for the research described in the article.

\section*{Acknowledgments}
This work was supported by the Talent Introduction Project of Xihua University (Grant No. Z241102), and the National Natural Science Foundation of China (Grant No. 12171339).

%The authors would like to thank the reviewers and the editors for their valuable comments and suggestions which
%contribute to improving the earlier versions of this paper.


\begin{thebibliography}{}

\bibitem{bauschke2019} Bauschke HH,  Burachik RS, Luke DR. Splitting Algorithms, Modern Operator Theory, and Applications. Springer; 2019.

\bibitem{bubeck} Bubeck S. Convex optimization: Algorithms and complexity. Found Trends Mach Learn 2015;8:231-357.

\bibitem{ChambolleP} Chambolle A, Pock T. A first-order primal-dual algorithm for convex problems with applications to imaging. J Math Imaging Vis 2011;40:120-145.

\bibitem{Goldstein2015} Goldstein T, Li M, Yuan X. Adaptive primal-dual splitting methods for statistical learning and image processing. Adv Neural Inf Process Syst 2015; 28.

\bibitem{LiML2019ACC} Lin Z, Li H, Fang C. Accelerated Optimization for Machine Learning. Springer; 2020.

\bibitem{Arrow} Arrow KJ, Hurwicz L, Chenery HB. Studies in linear and non-linear programming, With contributions by H. B. Chenery, S. M. Johnson, S. Karlin, T. Marschak, R. M. Solow. Stanford Mathematical Studies in the Social Sciences, vol. II, Stanford University Press, Stanford, Calif, 1958.

\bibitem{Korpelevich} Korpelevich GM. The extragradient method for finding saddle points and other problems. Ekon Mat Metody 1976;12:747-756.


\bibitem{Lions} Lions PL, Mercier B. Splitting algorithms for the sum of two nonlinear operators. SIAM J Numer Anal 1979;16:964-979.

\bibitem{Esser} Esser E, Zhang X, Chan TF. A general framework for a class of first order primal-dual algorithms for convex optimization in imaging science. SIAM J Imaging Sci 2010;3 (2010):1015-1046.


\bibitem{ChenSiam2014} Chen Y, Lan G, Ouyang Y. Optimal primal-dual methods for a class of saddle point problems. SIAM J Optim 2014;24:1779-1814.

\bibitem{Chambolle2016MP} Chambolle A, Pock T. On the ergodic convergence rates of a first-order primal-dual algorithm. Math Program 2016;159:253-287.


\bibitem{HeJMIV} He B, Ma F, Yuan X. An algorithmic framework of generalized primal-dual hybrid gradient methods for saddle point problems. J Math Imaging Vis 2017;58;279-293.

\bibitem{He2014siamPD} He B, You, Y,  Yuan X. On the convergence of primal-dual hybrid gradient algorithm. SIAM J Imaging Sci 2014;7:2526-2537.

\bibitem{Tran2022} Tran-Dinh Q. A unified convergence rate analysis of the accelerated smoothed gap reduction algorithm. {Optim Lett} 2022;16(4):1235-1257.

\bibitem{Tran2018} Tran-Dinh Q, Fercoq O, Cevher V. A smooth primal-dual optimization framework for nonsmooth composite convex minimization. SIAM J Optim 2018;28:96-134.

\bibitem{Zhu2022} Zhu Y, Liu D,  Tran-Dinh Q. New primal-dual algorithms for a class of nonsmooth and nonlinear convex-concave minimax problems. SIAM J Optim 2022;32(4);2580-2611.

\bibitem{Jiang2021}  Jiang F, Cai X,  Wu Z,  Han D. Approximate first-order primal-dual algorithms for saddle point problems. Math Comput 2011;90:1227-1262.

\bibitem{Jiang2021NA} Jiang F,  Wu Z, Cai X,  Han D. A first-order inexact primal-dual algorithm for a class of convex-concave saddle point problems. Numer Algorithms 2021;88:1109-1136 

\bibitem{RaschCH2020}  Rasch J, Chambolle A. Inexact first-order primal-dual algorithms. Comput Math Appl 2020;76:381-430.

\bibitem{Fercoq} Fercoq O, Bianchi P. A coordinate-descent primal-dual algorithm with large step size and possibly nonseparable functions. SIAM J Optim 2019;29:100-134.

\bibitem{MalitskySIAM} Malitsky Y, Pock T. A first-order primal-dual algorithm with linesearch. SIAM J Optim 2018;28:411-432.

\bibitem{chang2021} Chang X, Yang J. A golden ratio primal-dual algorithm for structured convex optimization. J Sci Comput 2021;87:1-26.


\bibitem{bai2020several} Bai J,  Li J,  Wu Z. Several variants of the primal-dual hybrid gradient algorithm with applications. Numer Math Theor Meth Appl 2020;13:176-199.


\bibitem{HeAMO2024} He X, Hu R, Fang YP. A second order primal–dual dynamical system for a convex–concave bilinear saddle point problem. Appl Math Optim 2024;89(2):30.


\bibitem{Boct2023COA} Bot RI, Csetnek, ER, Sedlmayer M. An accelerated minimax algorithm for convexconcave saddle point problems with nonsmooth coupling function. Comput Optim Appl 2023;86:925-966.

\bibitem{hamedani2021} Hamedani EY  Aybat NS. A primal-dual algorithm with line search for general convex-concave saddle point problems. SIAM J Optim 2021;31:1299-1329.

\bibitem{Mokhtari2020} Mokhtari A, Ozdaglar AE,  Pattathil S. Convergence rate of O(1/k) for optimistic gradient and extragradient methods in smooth convex-concave saddle point problems. SIAM J Optim 2020;30:3230-3251.

\bibitem{Tan2020} Tan C, Qian Y, Ma S, Zhang T. Accelerated dual-averaging primal-dual method for composite convex minimization. Optim Methods Softw 202;35:741-766.

\bibitem{Bai2023} Bai J, Chen Y. Generalized AFBA algorithm for saddle-point problems. Optimization Online, 2023.

\bibitem{HeY2015NM} He B,  Yuan X. On non-ergodic convergence rate of Douglas-Rachford alternating direction method of multipliers. Numer Math 2015;130:567-577.

\bibitem{Li2019} Li H,  Lin Z. Accelerated alternating direction method of multipliers: An optimal O(1/K) nonergodic analysis. J Sci Comput 2019;79:671-699.

\bibitem{Nesterov1983} Nesterov Y. A method for solving the convex programming problem with convergence rate $\mathcal{O}(1/k^2)$, 1983;269:543-547.

\bibitem{Nesterov2018} Nesterov Y. Lectures on Convex Optimization. Springer, 2018.

\bibitem{BeckIma} Beck A. Teboulle M. A fast iterative shrinkage-thresholding algorithm for linear inverse problems. SIAM J Imaging Sci 2009;2:183-202.

\bibitem{Tseng2008} Tseng P. On accelerated proximal gradient methods for convex-concave optimization, Technical report, University of Washington, Seattle, (2009).


\bibitem{LuoSco} Luo H, Zhang Z. A unified differential equation solver approach for separable convex optimization: splitting, acceleration and nonergodic rate. arXiv:2109.13467, 2023.

\bibitem{BotMP2022}  Bot¸RI,  Csetnek ER, Nguyen DK. Fast augmented Lagrangian method in the convex regime with convergence guarantees for the iterates. Math Program 2023;200:147-197.

\bibitem{Tran2020} Tran-Dinh Q, Zhu Y. Non-stationary first-order primal-dual algorithms with faster convergence rates. SIAM J Optim 2020;30:2866-2896.


\bibitem{Xu2017} Xu Y. Accelerated first-order primal-dual proximal methods for linearly constrained composite convex programming. SIAM J Optim. 2017;27:1459-1484.

\bibitem{HeNA} He X, Hu R, Fang YP. Inertial accelerated primal-dual methods for linear equality constrained convex optimization problems. Numer Algorithms 2022;90:1669-1690.


\bibitem{HeAuto} He X, Hu R, Fang YP. Fast primal-dual algorithm via dynamical system for a linearly constrained convex optimization problem. Automatica 2022;146:110547.

\bibitem{Luo2021Accer} Luo H. Accelerated primal-dual methods for linearly constrained convex optimization problems. arXiv:2109.12604. 2021.


\bibitem{HeADMM} He X, Huang NJ,  Fang YP. Accelerated linearized alternating direction method of multipliers with Nesterov extrapolation. arXiv:2310.16404. 2023.


\bibitem{Chambolle} Chambolle A, Dossal C. On the convergence of the iterates of the ``fast iterative shrinkage/thresholding algorithm''. J Optim Theory Appl 2015;166:968-982.


\bibitem{TranMPC} Tran-Dinh Q, Fercoq O, Cevher V. An adaptive primal-dual framework for nonsmooth convex minimization. Math Prog Comp 2020;12(3):451-491.


\bibitem{Wright2009} Wright, SJ,  Nowak RD, Figueiredo MA. Sparse reconstruction by separable approximation. IEEE Trans Signal Process 2009;57:2479–2493.



\end{thebibliography}
\end{document}